\documentclass[reqno,11pt]{amsart}
\usepackage[english]{babel}
\usepackage[latin1]{inputenc}
\usepackage{amsmath}\allowdisplaybreaks 
\usepackage{amssymb,amsfonts,amsthm,amscd}
\usepackage{fontenc,indentfirst,delarray,amsfonts,amsmath,amssymb}
\usepackage[T1]{fontenc}
\usepackage{mathrsfs}
\usepackage{color}

\usepackage{enumerate} 

\usepackage{cancel}
\usepackage{stmaryrd}
\usepackage{verbatim}
\usepackage{mathtools}
\font\black=cmbx10 \font\sblack=cmbx7 \font\ssblack=cmbx5 \font\blackital=cmmib10  \skewchar\blackital='177
\font\sblackital=cmmib7 \skewchar\sblackital='177 \font\ssblackital=cmmib5 \skewchar\ssblackital='177
\font\sanss=cmss12 \font\ssanss=cmss8 scaled 900 \font\sssanss=cmss8 scaled 600 \font\blackboard=msbm10
\font\sblackboard=msbm7 \font\ssblackboard=msbm5 \font\caligr=eusm10 \font\scaligr=eusm7 \font\sscaligr=eusm5
    
\font\bsymb=cmsy10 scaled\magstep2
\def\all#1{\setbox0=\hbox{\lower1.5pt\hbox{\bsymb
       \char"38}}\setbox1=\hbox{$_{#1}$} \box0\lower2pt\box1\;}
\def\exi#1{\setbox0=\hbox{\lower1.5pt\hbox{\bsymb \char"39}}
       \setbox1=\hbox{$_{#1}$} \box0\lower2pt\box1\;}

\newfam\bifam
\textfont\bifam=\blackital \scriptfont\bifam=\sblackital \scriptscriptfont\bifam=\ssblackital

\newfam\blfam
\textfont\blfam=\black \scriptfont\blfam=\sblack \scriptscriptfont\blfam=\ssblack

\newfam\bbfam
\textfont\bbfam=\blackboard \scriptfont\bbfam=\sblackboard \scriptscriptfont\bbfam=\ssblackboard

\newfam\ssfam
\textfont\ssfam=\sanss \scriptfont\ssfam=\ssanss \scriptscriptfont\ssfam=\sssanss
\def\sss#1{{\fam\ssfam\relax#1}}
\newfam\clfam
\textfont\clfam=\caligr \scriptfont\clfam=\scaligr \scriptscriptfont\clfam=\sscaligr

\def\pmb#1{\setbox0\hbox{${#1}$} \copy0 \kern-\wd0 \kern.2pt \box0}
\def\pmbb#1{\setbox0\hbox{${#1}$} \copy0 \kern-\wd0
      \kern.2pt \copy0 \kern-\wd0 \kern.2pt \box0}
\def\pmbbb#1{\setbox0\hbox{${#1}$} \copy0 \kern-\wd0
      \kern.2pt \copy0 \kern-\wd0 \kern.2pt
    \copy0 \kern-\wd0 \kern.2pt \box0}
\def\pmxb#1{\setbox0\hbox{${#1}$} \copy0 \kern-\wd0
      \kern.2pt \copy0 \kern-\wd0 \kern.2pt
      \copy0 \kern-\wd0 \kern.2pt \copy0 \kern-\wd0 \kern.2pt \box0}
\def\pmxbb#1{\setbox0\hbox{${#1}$} \copy0 \kern-\wd0 \kern.2pt
      \copy0 \kern-\wd0 \kern.2pt
      \copy0 \kern-\wd0 \kern.2pt \copy0 \kern-\wd0 \kern.2pt
      \copy0 \kern-\wd0 \kern.2pt \box0}

\mathchardef\za="710B  
\mathchardef\zb="710C  
\mathchardef\zg="710D  
\mathchardef\zd="710E  
\mathchardef\zve="710F 
\mathchardef\zz="7110  
\mathchardef\zh="7111  
\mathchardef\zvy="7112 
\mathchardef\zi="7113  
\mathchardef\zk="7114  
\mathchardef\zl="7115  
\mathchardef\zm="7116  
\mathchardef\zn="7117  
\mathchardef\zx="7118  
\mathchardef\zp="7119  
\mathchardef\zr="711A  
\mathchardef\zs="711B  
\mathchardef\zt="711C  
\mathchardef\zu="711D  
\mathchardef\zvf="711E 
\mathchardef\zq="711F  
\mathchardef\zc="7120  
\mathchardef\zw="7121  
\mathchardef\ze="7122  
\mathchardef\zy="7123  
\mathchardef\zf="7124  
\mathchardef\zvr="7125 
\mathchardef\zvs="7126 
\mathchardef\zf="7127  
\mathchardef\zG="7000  
\mathchardef\zD="7001  
\mathchardef\zY="7002  
\mathchardef\zL="7003  
\mathchardef\zX="7004  
\mathchardef\zP="7005  
\mathchardef\zS="7006  
\mathchardef\zU="7007  
\mathchardef\zF="7008  
\mathchardef\zW="700A  

\newcommand{\be}{\begin{equation}}
\newcommand{\ee}{\end{equation}}
\newcommand{\bea}{\begin{eqnarray}}
\newcommand{\eea}{\end{eqnarray}}
\newcommand{\beas}{\begin{eqnarray*}}
\newcommand{\eeas}{\end{eqnarray*}}

\newtheorem{defi}{Definition}
\newtheorem{thm}{Theorem}
\newtheorem{prop}{Proposition}
\newtheorem{lem}{Lemma}
\newtheorem{cor}{Corollary}
\newtheorem{exa}{Example}

\newtheorem{rema}{Remark}

\newcommand{\op}[1]{\!\!\mathop{\rm ~#1}\nolimits}

\DeclareMathOperator{\im}{im}
\DeclareMathOperator{\Hom}{Hom}
\DeclareMathOperator{\End}{End}
\DeclareMathOperator{\Aut}{Aut}
\DeclareMathOperator{\id}{id}

\newcommand{\degr}{\widetilde}

\DeclareMathOperator{\iHom}{\sss{Hom}}
\DeclareMathOperator{\iEnd}{\sss{End}}
\DeclareMathOperator{\iAut}{\sss{Aut}}

\DeclareMathOperator{\gt}{^{\zG\!\!\op{t}}\!}
\DeclareMathOperator{\transp}{^{\op{t}}\!}

\DeclareMathOperator{\tr}{tr}

\DeclareMathOperator{\gtr}{\zG\!\tr}

\DeclareMathOperator{\detinv}{\det\,\!\!^{-1}}

\DeclareMathOperator{\ber}{Ber}
\DeclareMathOperator{\gber}{\zG\!\ber}

\newcommand{\evp}{_{\bar{0}}} 
\newcommand{\odp}{_{\bar{1}}} 

\DeclareMathOperator{\qi}{i}
\DeclareMathOperator{\qj}{j}
\DeclareMathOperator{\qk}{k}

\newcommand{\cK}{\mathcal{K}}

\newcommand{\cL}{\mathcal{ L}}

\newcommand{\f}{\mathfrak}

\newcommand{\N}{\mathbb{N}}
\newcommand{\Z}{\mathbb{Z}}
\newcommand{\R}{\mathbb{R}}
\newcommand{\K}{\mathbb{K}}
\newcommand{\C}{\mathbb{C}}

\newcommand{\qH}{\mathbb{H}}

\newcommand{\I}{\mathbb{I}}

\newcommand{\la}{\langle}
\newcommand{\ra}{\rangle}
\newcommand{\lp}{\left(}
\newcommand{\rp}{\right)}
\newcommand{\lpq}{\left[}
\newcommand{\rpq}{\right]}

\newcommand*{\EnsQuot}[2]%
{\ensuremath{%
    #1/\!\raisebox{-.65ex}{\ensuremath{{#2}}}}}
\newcommand{\vect}{\mathbf}

\newcommand{\p}{\partial}

\newcommand{\ts}{\times}
\newcommand{\opl}{\oplus}
\newcommand{\ots}{\otimes}

\newcommand{\raa}{\rightarrow}

\newcommand{\bemaps}{\[ \begin{array}{lccc}}
\newcommand{\eemaps}{\end{array}\]}

\newcommand{\cds}{ \! \cdots  \!}

\newcommand{\Top}{\rule{0pt}{3ex}}

\begin{document}

\title{Cohomological Approach to the Graded Berezinian}
\date{}
\author{Tiffany Covolo}
\address{Tiffany Covolo, University of Luxembourg, Mathematics Research Unit,
6, rue Richard Coudenhove-Kalergi, L-1359 Luxembourg,
Grand Duchy of Luxembourg}
\email{tiffany.covolo@uni.lu}
\vspace{2mm}
\noindent
\maketitle \thispagestyle{empty}

	\begin{abstract}
We develop the theory of linear algebra over a $(\Z_2)^n$-commutative algebra ($n\in \N$), which includes the well-known super linear algebra as a special case ($n=1$). Examples of such graded-commutative algebras are the Clifford algebras, in particular the quaternion algebra $\qH$.
Following a cohomological approach, we introduce analogues of the notions of trace and determinant. Our construction reduces in the classical commutative case to the coordinate-free description of the determinant by means of the action of invertible matrices on the top exterior power, and in the supercommutative case it coincides with the well-known cohomological interpretation of the Berezinian.
	\end{abstract}

\section*{Introduction}

Remarkable series of algebras, such as the algebra of quaternions and,
more generally, \emph{Clifford algebras} turn out to be graded-commutative.
Originated in \cite{AM} and \cite{AM1}, this idea was developed in \cite{MO} and \cite{OM}.
The grading group in this case is $(\Z_2)^{n+1}$, where $n$ is the number of generators,
and the graded-commutativity reads as
\be \label{z2ncomm}a \cdot b=(-1)^{\la \degr{a} , \degr{b} \ra} b\cdot a\ee
where $\degr{a},\degr{b}\in (\Z_2)^{n+1}$ denote the degrees of the respective homogeneous elements $a$ and $b$, and $\la.,.\ra : (\Z_2)^{n+1}\ts (\Z_2)^{n+1} \to \Z_2$ is the standard scalar product of binary $(n+1)$-vectors
(see Section~\ref{Sec1}).
This choice of the graded-commutativity has various motivations.
First, it is the intuitive extension of the well-known superalgebra, which corresponds to this $(\Z_2)^n$-commutativity for $n=1$, $\la .\,,.\ra$ being in this case just classical multiplication.
Secondly, it was proved in \cite{OM} that
such $(\Z_2)^n$-commutativity is universal among graded-commutative algebras.
That is, \emph{if $\zG$ is a finitely generated Abelian group, then for an arbitrary
 $\zG$-graded-commutative algebra $A$ with graded-commutativity of the form
$a\cdot b= (-1)^{\zb(\degr{a},\degr{b})}b\cdot a$,  with $\zb:\zG\ts\zG \to \Z_2$  a bilinear symmetric map,
it exists $n\in \N$ such that $A$ is $(\Z_2)^n$-commutative (in the sense of (\ref{z2ncomm}))}.

First steps towards the $(\Z_2)^n$-graded version of linear algebra were done in \cite{COP}.
The notion of graded trace for all endomorphisms
and that of graded Berezinian for $0$-degree automorphisms
were introduced in the most general framework of
an arbitrary free module (of finite rank) over a $(\Z_2)^n$-commutative algebra.

In this paper, we develop a cohomological approach to the notion
of graded Berezinian and graded trace.
In the super case, this approach is originally due to O. V. Ogievetskii and  I. B. Penkov
(\cite{OP84}), but we will mostly refer to the description given in \cite{Man88}.
Similarly to this latter, we define a graded analogue of the Koszul complex
and the graded Berezinian module associated to a given
free module of finite rank. We believe this to be the first step towards the conception of a generalization of the Berezinian integral over a possible multi-graded (i.e. $(\Z_2)^n$-graded) manifold.

The paper is organized as follows.
We recall the basic notions of graded linear algebra in Section \ref{Sec1} and derive
the graded matrix calculus in Section \ref{Sec2}.
In Section \ref{Sec3}, we present our first main result, a cohomological
interpretation of the graded Berezinian.
In Section \ref{Sec4}, we give a similar description of the graded trace.
It is worth noticing that the cohomological description of the graded trace
of \emph{arbitrary even matrices} leads to interesting restrictions for the
grading group $(\Z_2)^n$, namely $n$ has to be \emph{odd}.
Furthermore, the parity changing operator has to be chosen in a canonical way
and corresponds to the element $(1,1,\ldots,1)$ of $(\Z_2)^n$.
\medskip

We have to note that there is an alternative approach to the generalization of superalgebras and related notions, which makes use of category theory. This approach  follows from results by Scheunert in \cite{Sch} (in the Lie algebras setting) and Nekludova (in the commutative algebra setting). An explicit description of the results of the latter can be found in \cite{SoS}.
This other method to treat the problem 
and its consequences in the $(\Z_2)^n$-commutative case, will be the object of a separate work.
\medskip
\medskip

{\bf Acknowledgments}.
The author is pleased to thank Norbert Poncin whose suggestions initiate the paper, Jean-Philippe Michel for enlightening discussions and Dimitry Leites
who made her aware of Nekludova's work.
The author is grateful to Valentin Ovsienko for the contributions he gave through the development and the finalization of the paper.

The author thanks the Luxembourgian NRF for support via AFR PhD grant 2010-1 786207.
\medskip
\medskip

\section{Graded linear algebra}\label{Sec1}

In this section, we give a brief survey of the main notions
of linear algebra over graded-commutative algebras.

Consider an Abelian group
$(\zG, +)$ endowed with a symmetric bi-additive map
$$
\la \;,\;\ra:\zG\times\zG\to\Z_2\,.
$$
We call $\zG$ the \emph{grading group}.
This group admits a natural splitting $\zG=\zG\evp\opl\zG\odp$, where
$\zG\evp$ is the subgroup characterized by $\la\zg,\zg\ra=0$ for all $\zg\in\zG\evp$,
and where  $\zG\odp$ is the set characterized by $\la\zg,\zg\ra=1$ for all $\zg\in\zG\odp$.
We call $\zG\evp$ and $\zG\odp$ the \emph{even subgroup} and \emph{odd part}, respectively.

In this paper, we restrict the considerations to the case $\zG\!=(\Z_2)^n$,
for some fixed $n\in \mathbb{N}$, equipped with the standard scalar product
$$
\la x,y\ra=
\sum_{1\leq{}i\leq{}n}x_iy_i
$$
of $n$-vectors, defined over $\Z_2$.
Our main example are the Clifford algebras equipped with the grading described in Example \ref{Clif} (see next section).
\medskip

\subsection{Graded-Commutative Algebras}

A {\it graded vector space} is a direct sum
$$V=\bigoplus_{\zg\in\zG}V^{\zg}$$
of vector spaces $V^{\zg}$ over a
commutative field $\K$ (that we always assume of characteristic 0).
An \emph{endomorphism} of $V$ is a $\K$-linear map from $V$ to $V$
that preserves the degree; we denote by $\End_{\K}(V)$ the space of endomorphisms.

\medskip
A {\it $\zG$-graded algebra} is an algebra $A$ which has a structure of a $\zG$-graded vector space
$
A=\bigoplus_{\zg \in\zG }A^{\zg}
$
such that the operation of multiplication respects the grading:
$$A^{\za}A^{\zb}\subset A^{\za+\zb}\;, \quad \forall\za,\zb\in\zG\;.$$
We always assume $A$ \emph{associative} and \emph{unital}.

An element $a\in A^{\zg}$ is called \emph{homogeneous of degree} $\zg$.
For every homogeneous element $a$, we denote by $\degr{a}$ its degree.
Because of the even-odd splitting of the grading group,
one also has
$$A=A\evp\oplus{}A\odp,$$
where $a\in{}A\evp$ or $A\odp$ if $\degr{a}$ is even or odd, respectively.
For simplicity, in most formulas below the involved elements  are assumed to be homogeneous. These expressions are then extended to arbitrary elements by linearity.

\medskip
A $\zG\!$-graded algebra is called $\zG\!$-\emph{commutative} if
\be \label{GCom} a\,b = (-1)^{\la \degr{a},\degr{b}\ra}\,b\,a .\ee
In particular, every odd element squares to zero.
\medskip

\begin{exa}\label{Clif}{\rm
As we have mentioned in the Introduction, a Clifford algebra $\op{Cl}_n$ of $n$ generators (over $\R$ or $\C$) is a $(\Z_2)^{n+1}$-commutative algebra (\cite{OM}). The grading is given on the generators $e_i$ of $\op{Cl}_n$ as follows
\beas \degr{e}_i = (0, \ldots, 0, 1, 0, \ldots 0,1) \in (\Z_2)^{n+1} \eeas
where $1$ is at the $i$-th and at the last position.
}\end{exa}
\medskip

\subsection{Graded Modules}

A \emph{graded left module} $M$ over a $\zG\!$-commutative algebra $A$
is a left $A$-module with a $\zG$-graded vector space structure $M=\bigoplus_{\zg \in\zG }M^{\zg}$ such that the $A$-module structure respects the grading, i.e.
$$A^{\za} M^{\zb}\subset M^{\za+\zb}\;, \quad \forall\za,\zb\in\zG\;.$$

The notion of \emph{graded right module over $A$} is defined analogously.
Thanks to the graded-commutativity of $A$,
a left $A$-module structure induces a compatible right $A$-module structure
given by
\be \label{LRmod}m\, a := (-1)^{\la\degr{a},\degr{m}\ra}a\, m\,,\ee
and vice-versa.
Hence, we identify the two concepts.

\medskip
An $A$-module is called \emph{free of total rank} $r\in \N$ if it admits a basis of r homogeneous elements $\{ e_s\}_{s=1, \ldots , r}$. In this case, every element $m\in M$ can be
decomposed in a basis either with left or right coefficients, which are clearly related through
(\ref{LRmod}).

\medskip
A \emph{morphism} of $A$-modules, is a map $\ell :  M \to N$ which is $A$-linear of degree zero
(i.e. $\ell(M^{\za}) \subset N^{\za}$, $\forall \za \in \zG$).
We will denote the set of such maps by $\Hom_A(M,N)$. We usually refer to this set as the categorical $\Hom$ since
$A$-modules with these degree-preserving $A$-linear maps
form a category {\tt Gr}$_\zG${\tt Mod}$_A$.

We remark that the $\Hom$ set of graded $A$-modules is not a graded $A$-module itself. The internal $\sss{Hom}$ of the category {\tt Gr}$_\zG${\tt Mod}$_A$ is
$$\iHom_A(M,N):=\bigoplus_{\zg\in \zG}\iHom^{\zg}_A(M,N)\;,$$
where each $\iHom^{\zg}_A(M,N)$ consists of $A$-linear maps $\ell:M\to N$ of degree $\zg$, that is additive maps satisfying
\be\ell(am)=
(-1)^{\la\zg,\,\degr{a}\ra}a\,\ell(m)
\quad
\big(\,\text{or equivalently, }
\ell(ma)=\ell(m)a \, \big) \label{right}\ee
$$ \mbox{ and } \qquad \ell(M^{\za})\subset{}N^{\za+\zg}\,.$$
The $A$-module structure of $\iHom_A(M,N)$ is given by
\be \label{iHomAmod} (a \cdot  \ell) (-):= (-1)^{\la\degr{\ell},\, \degr{a}\ra}\ell(a\cdot -)
\quad \big(\, \mbox{or equivalently, } (\ell \cdot a )(-)= \ell (a\cdot -) \,\big).
 \ee

\begin{rema}
{\rm
The categorical $\Hom$ coincides with the $0$-degree part of the internal $\iHom$, i.e.
$\Hom_A(M,N)=\iHom^0_A(M,N)$.
By abuse of notation, we also refer to the elements of the internal $\iHom$ as morphisms. To make clear the distinction between categorical and internal hom, we often add the adjective ``graded'' in the latter case.
}
\end{rema}

We define graded endomorphisms and graded automorphisms of an $A$-module $M$ by
\beas \iEnd_A(M):=\iHom_A(M,M)& \mbox{ and } & \iAut_A(M):=\{ \ell \in \iEnd_A(M) \; :\; \ell \mbox{ invertible } \} \;,\eeas
and their degree-preserving analogues by
\beas \End_A(M):=\iEnd^0_A(M)& \mbox{ and } & \Aut_A(M):=\iAut^0_A(M)\;.\eeas
In situations where it is not misleading, we will drop the subscript and just write $\iHom(M,N)$, $\iEnd(M)$, $\iAut(M)$, etc.
\medskip

The \emph{dual} of a graded $A$-module $M$ is the graded $A$-module $M^*:=\iHom(M,A)$. As for classical modules, if $M$ is free with basis $\{e_i\}_{i=1,\ldots ,r}$ then its dual module $M^*$ is also free of same rank. Its basis $\{\ze ^i \}_{i=1,\ldots ,r}$ is defined as usual by
$$ \ze^i(e_j)=\zd^i_j\;, \quad \forall \, i,j $$
where $\zd^i_j$ is the Kronecker delta.
Note that this implies $\degr{\ze^i}=\degr{e_i}$ for all $i$.

\subsection{Lie algebras, Derivations}

A \emph{$\zG$-colored Lie algebra} $A$ is a $\zG$-graded algebra in which the multiplication operation (denoted $[\;\cdot\, , \cdot\;]$) verify the following two conditions, for all homogeneous elements $a,b,c\in A$.
 \begin{enumerate}
 \item Graded skew-symmetry:  $$ [a,b]=-(-1)^{\la \degr{a},\degr{b}\ra} [b,a] \; ;$$
 \item Graded Jacoby identity:  $$ [a,[b,c]]=[[a,b],c]+(-1)^{\la \degr{a},\degr{b} \ra }[b,[a,c]]\;.$$
 \end{enumerate}

Natural examples of colored Lie algebras are $\zG$-graded associative algebras
with the graded commutator
\be \label{GrCommutator} [a,b]= ab-(-1)^{\la \degr{a},\degr{b}\ra} ba\;. \ee
If $A$ is a $(\Z_2)^n$-colored Lie algebra, its $0$-degree part $A^0$ is a classical Lie algebra.
\medskip

An homogeneous \emph{derivation} of degree $\zg$ of a $\zG$-graded algebra $A$ over a field $\K$, is a $\K$-linear map of degree $\zg$ $D \in \iEnd_{\K}^{\zg}(A)$
which verifies the \emph{graded Leibniz rule}
\beas D(ab)=D(a)\,b+(-1)^{\la \degr{a}, \zg \ra}a \, D(b) \eeas
for all homogeneous elements $a,b \in A$.

We denote the set of derivations of degree $\zg$ of $A$ by $\op{Der}^{\zg}(A)$. Then, the set of all graded derivations of $A$
$$\op{Der}(A):= \bigoplus_{\zg \in \zG}\op{Der}^{\zg}(A)$$ is a $\zG$-graded vector space. It is also an $A$-module, with the $A$-module structure given by $(aD)(x)=a\,D(x)$.
Moreover, considering composition of derivations, $\op{Der}(A)$
is also a colored Lie algebra for the commutator (\ref{GrCommutator}).

\subsection{Graded Tensor and Symmetric Algebras}
Let $A$ be a $\zG\!$-commutative algebra.

The \emph{tensor product} of two graded $A$-modules $M$ and $N$ can be defined
as follows. Let us forget, for the moment, the graded structure of $A$, seeing it simply as a non-commutative ring, and consider $M$ and $N$ respectively as right and left modules.  In this situation, the notion of tensor product is well-known (see \cite{K}), and the obtained object $M\ots_{A}N$ is a $\Z$-module.
Then, reconsidering the graded structure of the initial objects, we see that $M\ots_{A}N$ admits an induced
$\zG$-graded structure
 \be\label{degrots} M\ots_{A}N = \bigoplus_{\zg \in \zG} (M\ots_A N)^{\zg}= \bigoplus_{\zg \in \zG} \bigoplus_{\za + \zb =\zg} \left\{ \sum m \ots_A n \; \left| \;  m\in M^{\za}, n \in N^{\zb}\right. \right\}\;.\ee
Moreover, because of the actual two-sided module structure of both $M$ and $N$, the resulting object $M\ots_{A}N$ have also right and left $A$-module structures, which are by construction compatible (in the sense of (\ref{LRmod})).

As usual, tensor product of graded $A$-modules can be characterized as a universal object. 
All classical results and constructions related to the tensor product can then be transferred to the graded case without major difficulties.\medskip

Set $M^{\ots k}=M \ots_A \ldots \ots_A M$ ($n$ factors $M$, $n\geq 1$) and $M^{\ots 0}:= A$. The graded $A$-module
$$ T^{\bullet}_A M := \bigoplus\nolimits_{k\in \N} M^{\ots k} $$
is an associative graded $A$-algebra, called {\it graded tensor algebra}, with multiplication
 $$ \ots_A : M^{\ots r} \ts M^{\ots s} \to M^{\ots r} \ots_A M^{\ots s}\simeq M^{\ots (r+s)}\;.$$
The graded $A$-algebra $T^{\bullet}_A M$ is in fact bi-graded: it has the classical $\N$-grading (given by the number of factors in $M$) that we call
\emph{weight}, and an induced $\zG$-grading (see (\ref{degrots})) called \emph{degree}.

Taking the quotient of $T^{\bullet}_A M$ by the ideal $J_{S}$ generated by the elements of the form
$$ m\ots m' -(-1)^{\la \degr{m}, \degr{m'}\ra} m' \ots m , \quad \forall m,m' \in M $$
we obtain a  $\zG\!$-commutative $A$-algebra
$  S_A^{\bullet}(M) $
called the \emph{graded symmetric algebra}.

As their classical analogues, both these notions satisfy universal properties.
\subsection{Change of Parity Functors}
		
Unlike the classical super case (i.e. $\zG=\Z_2$), in general we have many different {\it parity reversion functors}.

For every $\zp \in \zG\odp$, we specify an endofunctor of the category of modules over a $\zG\!$-commutative algebra $A$
	\bemaps
 \zP: & \mbox{{\tt Gr}$_\zG${\tt Mod}$_A$} & \longrightarrow &\mbox{{\tt Gr}$_\zG${\tt Mod}$_A$} \\
               &   M &  \mapsto & \zP M   \\
               & \Hom_A(M,N) \ni f & \mapsto & f^{\zP} \in\Hom_A(\zP M, \zP N)
    \eemaps
The object $\zP M$ is defined by
\beas (\zP M)^\za := M^{\za+\zp}\;, \quad \forall\za \in \zG\;,\eeas
and graded $A$-module structure 
\beas \zP(m+m'):=\zP m+\zP m'& \mbox{ and } &  \zP(am):=(-1)^{\la \zp ,\, \degr{a} \ra} \, a\;\zP m \;.  \eeas
The morphism $f^{\zP}\in \Hom_A(\zP M,\zP N)$ is defined by $$ f^{\zP}\,(\zP m):=\zP\lp f(m)\rp \;.$$

Clearly, the map $\zP$ which sends an $A$-module $M$ to the $A$-module $\zP M$ is an $A$-linear map of degree $\zp$, i.e. $\zP \in \iHom^{\zp}(M,\zP M)$.

\section{Graded Matrix Calculus}\label{Sec2}


A graded morphism $t :M \to N$ of free $A$-modules of total rank $r$ and $s$ respectively, can be represented by a matrix over $A$.
Fixing a basis $\{e_i\}_{i=1,\ldots ,r}$ of $M$ and  $\{h_j\}_{j=1,\ldots ,s}$ of $N$
and considering elements of the modules as column vectors of right coordinates
\beas M\ni m= \sum e_j \, a^j &\quad  \simeq \quad & \vect{m}=\lp\begin{array}{c} a^1 \\ a^2 \\ \vdots \end{array}\rp\,, \eeas
the graded morphism is defined by the images of the basis vectors
$$ t(m)= \sum_j t(e_i)\,  a^i = \sum_{i,j} h_j\,  t^j_{\;\, i} \, a^i  $$
Hence, applying $t$ corresponds to left multiplication by the matrix $T=(t^j_{\;\, i}) \in \sss{M}(s\ts r ; A )$, 
$$ t(m)\simeq T \vect{m}.$$
We have a similar description when considering elements of the modules
 as row vectors of left coordinates.
In this paper, we choose the first approach.
This choice is justified by the fact that the graded morphisms are easier to handle  when one consider the right module structure, see (\ref{right}).

In what follows, the graded modules are implicitly assumed to be free, except when explicitly stated.

\subsection{Case $\zG=(\Z_2)^n$ }

$(\Z_2)^n$ is a additive group of finite order $N:=2^n$ and we can enumerate its elements following the \emph{standard order}: the first $q:=2^{n-1}$ elements being the even degrees ordered by lexicographical order, and the last ones being the remaining odd degrees, also ordered lexicographically.
E.g. $(\Z_2)^2=\{ (0,0), (1,1), (0,1),(1,0) \!\}$.
In the following, we denote $\zg_{i}$ the $i$-th element of $(\Z_2)^n$ with respect to this standard order.

This allows to re-order the basis of the considered graded $A$-modules following the degrees of the elements. We call a basis ordered in this way a \emph{standard basis}. From now on, we only consider this type of basis.

The \emph{rank} of a free graded module $M$ over a $(\Z_2)^n$-commutative algebra is then a $N$-tuple $\vect{r}=(r_1, \ldots , r_N) \in \N^N$ where each $r_i$ is the number of basis elements of degree $\zg_i$.
Hence, a standard basis $\{e_j\}_{j=1, \ldots ,r}$ of a graded $A$-module of rank $\vect{r}$ is such that the first $r_1$ elements are of degree $\zg_1$, the following $r_2$ elements are of degree $\zg_2$, etc.

Consequently, the matrix corresponding to an homogeneous graded morphism $t :M \to N$ of free $A$-modules of ranks $\vect{r}$ and $\vect{s}$ respectively, writes as a block matrix
\begin{equation}
\label{MatT}
T=\left(
\begin{array}{c|c|c}
T_{11} \Top &\; \ldots \; &T_{1N}\\[6pt]
\hline
\ldots \Top &\ldots&\ldots\\[3pt]
\hline
&&\\[-4pt]
T_{N1} &\ldots&T_{NN} \\[3pt]
\end{array}
\right)\,,
\end{equation}
where each block $T_{uv}$ (of dimension $s_u \ts r_v$) have homogeneous entries of the same degree. This latter is given by $\zg_u+\zg_v+ \degr{t} $, i.e. it depends on both the position $(u,v)$ of the block 
and the degree of the matrix, which is by definition the degree of the corresponding graded morphism $t$.
We will denote by
$\sss{M}(\vect{s},\vect{r};A)=\bigoplus_{\zg \in(\Z_2)^n}\sss{M}^{\zg}(\vect{s},\vect{r};A) $
the space of such matrices, also called \emph{graded matrices}.

\begin{exa}
{\rm
As a particular case of Clifford algebras, the algebra of quaternions $\qH$ is a $(\Z_2)^3$-commutative algebra (\cite{MO}).
We assign to the generators $1$, $\qi$, $\qj$ and $\qk$ of $\qH$ a degree following the standard order of $(\Z_2)^3$, i.e.
 $$
 \begin{array}{ccccc}
     \degr{1} := (0,0,0)\;, & \degr{\qi}:=(0,1,1)\;,& \degr{\qj}:=(1,0,1) & \mbox{ and } & \degr{\qk}:=(1,1,0)\;.
 \end{array}
 $$
Note that, as every other Clifford algebra (with the gradation given as in Example \ref{Clif}), it is purely even i.e. graded only by the even subgroup of $(\Z_2)^n$.

Every graded endomorphism of a free $\qH$-module $M$ of rank $\vect{r}$ then corresponds to a matrix in $\sss{M}\lp\vect{r};\qH \rp$, that is
a quaternionic matrix.
}
\end{exa}

 The $1$-to-$1$ correspondence
\beas \sss{M}^{\zg}(\vect{s},\vect{r};A) \simeq \iHom^{\zg}(M,N) \;,& \forall \zg \in(\Z_2)^n\;,\eeas
permits to transfer the associative graded $A$-algebra structure of $\iHom(M,N)$ to $\sss{M}(\vect{s},\vect{r};A)$.
Thus, we get the usual sum and multiplication of matrices, as well as an unusual multiplication of matrices by scalars in $A$ defined as follows.
\beas
aT=\left(
\begin{array}{c|c|c}
(-1)^{\la \degr{a}, \zg_1 \ra}aT_{11} \Top &\; \ldots \; & (-1)^{\la \degr{a}, \zg_1 \ra}aT_{1N}\\[6pt]
\hline
\ldots \Top &\ldots&\ldots\\[3pt]
\hline
&&\\[-4pt]
(-1)^{\la \degr{a}, \zg_N \ra}aT_{N1} &\ldots&(-1)^{\la \degr{a}, \zg_N \ra}aT_{NN} \\[3pt]
\end{array}
\right)\,,
\eeas
i.e. the $(i,j)$-th entry lying in the $(u,v)$-block is given by
$$ (aT)^i_{\;j} = (-1)^{\la \degr{a}, \zg_u \ra}a \, t^i_{\;j} \,.$$
Note that the sign appearing here is a direct consequence of (\ref{iHomAmod}).

In particular, graded endomorphisms $\iEnd(M)$ of a free graded $A$-module $M$ (of finite rank $\vect{r}$) can be seen as square graded matrices $\sss{M}(\vect{r};A):=\sss{M}(\vect{r},\vect{r};A)$. With the commutator (\ref{GrCommutator})
it is another example of $(\Z_2)^n$-colored Lie algebra.
\medskip

\subsection{Graded Transpose}\label{GradedTranspose}

As in the previous section, let $M$ and $N$ be
two graded $A$-modules of ranks $\vect{r}$ and $\vect{s}$ respectively.

The graded transpose $\gt T$ of a matrix $T \in \sss{M}(\vect{s},\vect{r};A) $,
that corresponds to $t \in \iHom(M,N)$, is defined as the matrix corresponding to the transpose $t^*\in \iHom(N^*, M^*)$ of $t$. For simplicity, we suppose $t$ to be homogeneous of degree $\degr{t}\;$.

We recall that the dual graded $A$-module of $M$ is $M^*:=\iHom(M,A)$, so that the dual morphism $t^*$  is naturally defined, for all $n^* \in N^*$ and all $m\in M$, by
\be \label{gt} \lp  t^*(n^*), m\rp = (-1)^{\la \degr{t},\degr{n^*} \ra} (n^*, t(m)) \ee
where $(-,-)$ denotes the evaluation of the involved morphisms on the corresponding source-module element.

Let $\{ e_k\}_{k=1,\ldots ,r}$ (resp. $\{h_l\}_{l=1,\ldots ,s}$) be the basis of 
$M$ (resp. $N$) and let $\{\ze^k \}_{k=1,\ldots ,r}$ (resp. $\{\zh^l\}_{l=1,\ldots ,s}$) the corresponding dual basis.

\begin{prop}
The \emph{graded transpose} $\gt T$ of a matrix $T=(T_{uv}) \in \sss{M}^{\degr{t}}(\vect{s},\vect{r};A)$ (considering here its block form (\ref{MatT})) is given by
$$ \lp\gt T\rp_{vu}= (-1)^{\la \zg_u+\zg_v, \degr{t} + \zg_v \ra} \transp T_{uv} .$$
where $\transp \;$ is the classical transpose.
\end{prop}
\begin{proof}
From the definition (\ref{gt}) of $t^*$, we have that
\beas
\lp  t^*(\zh^j), e_i\rp = (-1)^{\la \degr{t},\degr{h_j} \ra} (\zh^j, t(e_i))
  											 = (-1)^{\la \degr{t},\degr{h_j} \ra}  \sum_k  (\zh^j, h_k )t^k_{\;i}
 												 = (-1)^{\la \degr{t},\degr{h_j}\ra} t^j_{\;i}\;.
\eeas
On the other hand, denoting $t^{*\,j}_i$ the $(i,j)$-entry of the matrix $\gt T$, 
we have
\beas
\lp  t^*(\zh^j), e_i\rp = \sum_k (\ze^k t^{*\,j}_k, e_i)
												 = \sum_k (-1)^{\la \degr{t^{*\,j}_k}, \degr{e}_i\ra}(\ze^k , e_i)t^{*\,j}_k
												 = (-1)^{\la \degr{t^{*\,j}_i}, \degr{e}_i\ra} t^{*\,j}_i\;.
\eeas
Thus, for all $i,j$,
\be \label{transpSign} t^{*\,j}_i = (-1)^{\la \degr{t^{*\,j}_i}, \degr{e}_i\ra+\la \degr{t},\degr{h_j}\ra } t^j_{\;i}\ee
In particular, $\degr{ t^{*\,j}_i}= \degr{t^j_{\;i}}=\degr{t}+\degr{e_i}+\degr{h_j}$ for all $i,j$, so that the transpose morphism $t^*$ is also  homogeneous of same degree $\degr{t}$. Hence, (\ref{transpSign}) rewrites as
$$ t^{*\,j}_i = (-1)^{\la  \degr{t}+\degr{e}_i,\degr{h_j}+\degr{e}_i\ra } t^j_{\;i}  $$
and the result follows.
\end{proof}

In the super case (i.e. $n=1$), the graded transpose coincide with the well-known super transpose.

\medskip
It is easily verified by straightforward computations, that the operation of graded-transposition satisfy the following familiar property.
\begin{cor}\label{lemma}
For any pair of homogeneous square graded matrices $S,T \in \sss{M}(\vect{r};A)$, of degrees $s$ and $t$ respectively, we have that
$$ \gt\lp ST\rp=(-1)^{\la s,t \ra}\gt T \,\gt S \;.$$
Consequently, we also have that
\bea\label{lem} [\gt S, \gt T ]=-\gt\, [ S,T] \;.\eea
\end{cor}

\medskip

\subsection{Graded Trace}
As in the classical context, for any graded $A$-module $M$ there is a natural isomorphism of $A$-modules
\be \label{clsqiso} M^*\!\ots_A M \simeq \iEnd(M)\,.\ee
It is given by reading a tensor $\za \ots m \in M^*\!\ots_A M$ as the endomorphism
\beas
\za \ots m\; :& M \ni m' \; \mapsto \; (-1)^{\la \degr{m},\degr{m'} \ra} \za(m')m\; \in  M
\eeas

In the case where $M$ is free (of rank $\vect{r}$), this endomorphism is represented by a matrix $T=(t^i_{\;j}) \in \sss{M}(\vect{r};A)$, where the $(i,j)$-th entry is
$$ t^i_{\;j}=(-1)^{\la \degr{m},e_j\ra+\la \degr{\za_j},\degr{e}_j+\degr{e}_i \ra}\za_j m^i $$

The above isomorphism permits to define the \emph{graded trace} of the matrix corresponding to the endomorphism $\za \ots m$ as its contraction $\za(m)$ (as a $(1,1)$-tensor).

\begin{defi}
{\rm
The \emph{graded trace} of an homogeneous matrix $T=(T_{uv}) \in \sss{M}^{\degr{t}}(\vect{r};A)$ (considering here its block form (\ref{MatT})) is defined as
\be\label{trace} \gtr(T) := \sum_u (-1)^{\la \zg_u+\degr{t},\zg_u\ra} \tr(T_{uu}) \ee
where $\tr$ denotes the classic trace of a matrix.
}
\end{defi}

It is proved in \cite{COP} that
\textit{
$\op{\zG tr}:\sss{M}(\mathbf{r}; A)\to A$
is the unique (up to multiplication by a scalar of degree $0$)
homomorphism of $A$-linear colored Lie algebras.}
\medskip

                \subsection{Graded Berezinian}\label{GBER}
Fixing a standard basis $\{e_i\}_{i=1, \ldots ,r}$ of a free graded $A$-module $M$ permits to represent degree-preserving automorphisms of this module as invertible $0$-degree matrices, the group of which we denote by $\sss{GL}^0(\vect{r};A)$.
In \cite{COP} we have introduced the notion of \emph{Graded Berezinian} for this type of  matrices. Let us recall the main result of \cite{COP}.

\textit{
There is a unique group homomorphism
$$ \gber: \sss{GL}^{0}(\mathbf{r}; A) \to (A^{0})^{\times} $$
such that:
\begin{enumerate}
        \item For every block-diagonal matrix $X \in\sss{GL}^{0}(\mathbf{r}; A)$,
        $$ \gber(X)=\prod_{u=1}^{q} \det(X_{uu})\, \cdot \!\!\! \prod_{u=q+1}^{N} \detinv(X_{uu})\,.$$
    \item The image of any lower $($resp., upper$)$  block-unitriangular matrix in $\sss{GL}^{0}(\mathbf{r}; A)$ equals $1\in(A^0)^{\times}$.
        \end{enumerate}
}\medskip
Here, $ (A^{0})^{\times}$ denotes the invertible elements of the $0$-degree part of $A$, and $q:= N/2=2^{n-1}$.\medskip

Note that, similarly to the classical Berezinian, $\gber$ is defined only for $0$-degree
invertible matrices.
In the particular case where the module $M$ is graded by the even part $(\Z_2)^n\evp$,
the function $\gber$ is a polynomial.
Moreover, if $A=\qH$ then $\gber$ coincides with the classical Diedonn\'e determinant (restricted to these type of quaternionic matrices), see \cite{COP}.

\section{Cohomological Definition of the Graded Berezinian} \label{Sec3}

In this section, we define the Graded Berezinian module and describe its
cohomological interpretation. We obtain the function $\gber$ from the action of the group of degree-preserving automorphisms.
This construction generalizes the one described in \cite{Man88} for the classical Berezinian.

\subsection{Graded Berezinian Module}
Given a free $A$-module $M$, the corresponding  {\it graded Berezinian module} $\gber(M)$ is the free $A$-module of total rank $1$ built up from formal basis elements $\op{B}(\{e_i\})$ for each standard basis $\{e_i\}_{i=1,\ldots ,r}$ of $M$.
The transformation law induced by a change of basis $e'_i=e_j \, t^{j}_{\; i}$ in $M$ (of transition matrix $T=(t^j_{\;i})$) is given by
\be \label{trel1}\op{B}(\{e'_i\})=\op{B}(\{e_i\})\gber(T).\ee

Hence, intuitively, $\gber(M)$ is the free $A$-module of total rank $1$ which is functorial with respect to $\vect{0}$-degree automorphisms of $A$-modules and if $M$ is concentrated in only one even degree (i.e. is just a classical module) it coincides with the classical determinant module
$\op{Det}(M):=\bigwedge^{\op{top}}M$.

The above description of the module $\gber(M)$ is quite abstract.
In the text section, we will present an explicit cohomological construction of this module.

\subsection{Cohomological Construction}

Consider the graded-commutative algebra $S_A^{\bullet}(\zP M \opl  M^*)$, where $M$ is a (free) graded $A$-module of rank $\vect{r}$.

$S_A^{\bullet}(\zP M \opl  M^*)$ is the $(\Z _2)^n$-graded-commutative algebra of
polynomials in the graded variables $\zP e_i$ and $\ze^i$ with coefficients in $A$.
We define the operator $d$ to be left multiplication by the following element of $S_A^{2}(\zP M \opl  M^*)$:
\begin{equation}
\label{dEq}
d=\sum_i \zP e_i\, \ze^i,
\end{equation}
that we also denote by $d$, by abuse of notation.
This choice of $d$ is natural since $d|_M\equiv \zP$.

This defines a cochain complex $\cK ^{\bullet}:=\big( S_A^{\bullet}(\zP M \opl  ~M^*),\,d \,\big)$, thanks to the following result.

\begin{prop}\label{differential}
The operator $d$ is independent of the choice of the basis and $d^2=0$.
\end{prop}

\begin{proof}
Let us consider a transformation matrix $T=(t^i_{\;j}) \in  \sss{GL}^0(\vect{r};A)$ from the basis $\{ e_i\}_{i=1,\ldots ,r}$ to the basis $\{e'_i\}_{i=1,\ldots ,r}$ (both standard, $\degr{e'}_i=\degr{e}_i$ $\forall i$), i.e.
$
e'_i= \sum_j e_j\, t^j_{\,\; i}\,.
$
The transformation matrix between the induced basis $\{\zP e_i\}_{i=1,\ldots ,r}$ and $\{\zP e'_i\}_{i=1,\ldots ,r}$ of $\zP M$ (resp. between the dual basis $\{\ze^i\}_{i=1,\ldots ,r}$ and $\{\ze'^i\}_{i=1,\ldots ,r}$ of $M^*$) is then $T$ (resp. $\gt (T^{-1})$), i.e.
\bea \zP e'_i= \sum_j \zP e_j t^j_{\,\; i} & \mbox{ and }& \ze'^i=\sum_k \ze^k \f{t}^{*\,i}_k= \sum_k \ze^k (-1)^{\la \degr{e}_k+\degr{e}_i\,,\, \degr{e}_i \ra}\f{t}^i_{\,\; k} \label{groupaction} \eea
where $\f{t}^i_{\,\; k}$ denotes the $(i,k)$-th entry of $T^{-1}$.

Hence, we have
\beas
\sum_{i}\zP e'_i \; \ze'^i &=& \sum_{i,j,k} \zP e_j \, t^j_{\;\,i} \; \ze^k (-1)^{\la \degr{e}_i+\degr{e}_k\,,\, \degr{e}_k \ra} \f{t}^i_{\;\, k}\\
&=& \sum_{j,k}\zP e_j \; \ze^k \lp \sum_i (-1)^{\la \degr{e}_i+\degr{e}_k\,,\, \degr{e}_k \ra+\la \degr{e}_j +\degr{e}_i\,,\,\degr{e}_k \ra} t^j_{\;\,i} \, \f{t}^i_{\;\, k}\rp \\
&=& \sum_{j,k}\zP e_j \; \ze^k (-1)^{\la \degr{e}_k\,,\,\degr{e}_k+\degr{e}_j  \ra}\lp \sum_i t^j_{\;\,i} \, \f{t}^i_{\;\, k}\rp \\
        &=& \sum_{j,k}\zP e_j \; \ze^k (-1)^{\la \degr{e}_k\,,\,\degr{e}_k+\degr{e}_j  \ra} \zd^j_k = \sum_{j}\zP e_j \; \ze^j
\eeas
so that $d$ is well-defined.

The fact that $d$ squares to zero is easily checked by direct computation, using the graded-commutativity.
\end{proof}

To lighten the notation, we will denote by $x_i$ the even elements in $\{\zP e_i\}\cup\{\ze^i\}$ and by $\zx _i$ the odd ones (up to sign). More precisely, with $r':=\sum_{i=1}^{q}r_i$ indicating the number of basis elements of $M$ of even degree, we set
\bea \label{changenotation}
x_i:= \begin{cases} \ze^i & \mbox{ if } 1\leq i \leq r' \\ \zP e_i & \mbox{ if }  r'+1 \leq i \leq r \end{cases}
& \mbox{ and } &
\zx_i:= \begin{cases} \zP e_i & \mbox{ if } 1\leq i \leq r'\\ -(-1)^{\la \degr{e}_i,\zp \ra}\ze^i & \mbox{ if }  r'+1 \leq i \leq r  \end{cases}
\eea
By construction, we still have that $\degr{\zx}_i=\degr{x}_i+\zp$ for all $i$.

With this notation, the differential $d$ corresponds to left multiplication by $\sum_i \zx_i x_i$,
and $S_{A}^{\bullet}(\zP M \opl  M^*)$ is now viewed as the $(\Z_2)^n$-commutative algebra $A[x,\zx]$ of polynomials in the $(\Z_2)^n$-graded variables $x$-s and $\zx$-s.
Let us stress the fact that the product of polynomials is here the one which is naturally induced by the graded sign rule (\ref{GCom}).

The elements of the cochain complex $\cK ^{\bullet}$ at $k$-th level ($k\geq 0$) are defined as the polynomials in $A[x,\zx]$ with $k$-th total power degree in $\zx$. The element $\zx_{1}\cdots \zx_{r}$ is a cocycle. Indeed, we have that
$$
d(\zx_{1}\cdots \zx_{r})=
\sum_i (-1)^{\left\la\degr{\xi}_i,\,\degr{x}_i+\sum_{k<i}\degr{\xi}_k\right\ra}\,
x_i \zx_{1}\cdots \zx_{i-1} \zx_i^2 \zx_{i+1} \cdots \zx_{r}
=0
$$
since $\zx_i$ are odd and hence squares to $0$.
By this same observation,  $\cK ^{k}=0 $ for all $k>r$.
Hence, this cocycle will play the analogous role of the classical ``top element''.

  \begin{prop}\label{cohom}
$$ H^k(\cK ^{\bullet})= \begin{cases}\quad\quad 0 & \mbox{ if } k\neq r \\  [\zx_1\cds\zx_r]\cdot A & \mbox{ if }k=r \end{cases}  $$
\end{prop}

\begin{proof}
Let us consider the operator $\zr=\sum_i \frac{\p}{\p {x_i}} \frac{\p}{\p {\zx_i}} $, where $ \frac{\p}{\p {x_i}}\;,\,  \frac{\p}{\p {\zx_i}}$ are graded homogeneous partial derivations, i.e.
$$
\begin{array}{rclrcl}
\frac{\p}{\p {x_i}} x_{j}&=& \zd^i_{j} \;\;,\,  \quad &\quad  \frac{\p}{\p {x_i}} \zx_{j}&=&0 \\[4pt]
\frac{\p}{\p {\zx_i}} x_{j}&=&\; 0 \;\;\;,\,  \quad & \quad \frac{\p}{\p {\zx_i}} \zx_{j}&=& \zd^i_{j} .
\end{array}
$$
for all indices $i,j$. Note that the respective degrees are
 \beas \degr{\frac{\p}{\p {x_i}}}=\degr{x}_i& \mbox{ and }& \degr{\frac{\p}{\p {\zx_i}}}=\degr{\zx}_i=\degr{x}_i+\zp\;.\eeas

Let us compute $\lpq \zr\;,\, d \rpq $ where $\lpq \;.  \;,\,  .\;\rpq $ is the $(\Z _2)^n$-commutator (\ref{GrCommutator}).

 \beas
 \lpq \zr\;,\, d \rpq  & = & \sum_{i\;,\, j}\lpq \frac{\p}{\p {x_i}} \frac{\p}{\p {\zx_i}}\,,\, \zx_jx_j\rpq \\
 					&= & \sum_{i\;,\, j}\frac{\p}{\p {x_i}}\lpq \frac{\p}{\p {\zx_i}}\,,\, \zx_jx_j\rpq +\sum_{i\;,\, j}(-1)^{\la \degr{x}_i+\zp\,,\, \zp \ra}\lpq \frac{\p}{\p {x_i}}\,\, \zx_jx_j\rpq \frac{\p}{\p {\zx_i}}\\
 					&= &  \sum_{i\;,\, j}\frac{\p}{\p {x_i}}\lpq \frac{\p}{\p {\zx_i}}\,,\, \zx_j\rpq x_j+\sum_{i\;,\, j}(-1)^{\la \degr{x}_i+\zp \,,\, \degr{x}_j+\zp \ra}\frac{\p}{\p {x_i}}\zx_j\lpq \frac{\p}{\p {\zx_i}}\;,\, x_j\rpq \\
 					&& - \sum_{i\;,\, j}(-1)^{\la \degr{x}_i\,,\, \zp \ra}\lpq \frac{\p}{\p {x_i}}\;,\, \zx_j\rpq x_j\frac{\p}{\p {\zx_i}}-\sum_{i\;,\, j}(-1)^{\la \degr{x}_i\,,\, \zp \ra+\la \degr{x}_i\,,\,  \degr{x}_j+\zp\ra}\zx_j\lpq \frac{\p}{\p {x_i}}\;,\, x_j\rpq \frac{\p}{\p {\zx_i}}
 \eeas
But, by construction,
\beas
&\lpq \frac{\p}{\p {x_i}}\;,\, x_j\rpq =\zd^i_{j} \;\;,\,  \quad & \quad  \textstyle{\lpq \frac{\p}{\p {x_i}}\;,\, \zx_j\rpq =0 \;,}\\
&\lpq \frac{\p}{\p {\zx_{i}}}\;,\,  x_{j}\rpq =\;0 \; \;\;,\,  \quad   & \quad \textstyle{\lpq \frac{\p}{\p {\zx_i}}\;,\, \zx_j\rpq =\zd^i_{j} \;,}
\eeas
so that we have
 \beas
 \lpq \zr\;,\, d \rpq 	&= &  \sum_{i}\frac{\p}{\p {x_i}}x_i-\sum_{i}(-1)^{\la\degr{x}_i\,,\, \degr{x}_i\ra}\zx_i\frac{\p}{\p {\zx_i}} \\
 & = & \sum_{i}\lp \id + (-1)^{\la \degr{x}_i\,,\, \degr{x}_i \ra} x_i \frac{\p}{\p {x_i}}\rp -\sum_{i} (-1)^{\la\degr{x}_i\,,\, \degr{x}_i\ra} \zx_i\frac{\p}{\p {\zx_i}}\\
 &=& r \id +\sum_{i}x_i \frac{\p}{\p {x_i}} -\sum_{i}\zx_i\frac{\p}{\p {\zx_i}}
 \eeas
since each $x_i$ have even degree\;,\,  i.e. $\la\degr{x}_i\;,\, \degr{x}_i\ra=0$ for all $i$.

Now,  if $P$ is a homogeneous monomial in $\cK^k$,  i.e.
$$P=\zx^\za x^\zb a_{\za\zb}$$
 with $\za \in \{0,1\}^r$ such that $\left|\za\right|:=\sum_i \za_i =k$\;,\,  $\zb\in \N^r$ and $a_{\za\zb}$ a homogeneous element of $A$\;,\,  we have for every $i$ that
 $$ x_i\frac{\p}{\p {x_i}} (P) = \begin{cases} \zb_i\,  P & \mbox{ if $\zb_i\neq 0$}\\ 0 & \mbox{ if $\zb_i= 0$} \end{cases} $$
and
$$ \zx_i\frac{\p}{\p {\zx_i}} (P) = \begin{cases} P & \mbox{ if $\za_i=1$}\\ 0 &  \mbox{ if $\za_i=0$} \end{cases}$$
so that
$$ \lpq \zr\;,\, d \rpq (P)= \lp r+ \left| \zb\right| - k \rp \, \id(P) \;.$$
In fact, we only have to consider $0\leq k \leq r$. 
Hence,  $c:=r + \left| \zb\right| - k$ is zero if and only if $k=r$ and $\zb=0$.
It follows that,  for $k\neq r$,   we have a cochain homotopy between the identity $\op{id}$ and the zero map. It is given by  $\zr/c$ on monomials of the same form of $P$. 
We  conclude that $H^k(\cK^{\bullet})=0$ for all $k\neq r$.

It remains to consider the case when $k=r$.
By definition,  $$H^r(\cK^{\bullet})= \EnsQuot{\ker(d:\cK^r\to \cK^{r+1})}{\, \im(d:\cK^{r-1}\to \cK^r )}$$
where $\cK^r=\{\zx_1\cds\zx_r \,  Q \; \textbar \; Q \in A\lpq x\rpq  \}$ and $\cK^{r+1}=0$.
Hence,  $\ker(d:\cK^r\to \cK^{r+1})=\cK^r$. On the other hand,  by direct computation (e.g. apply $d$ on an element of the form $\zx_{\za_1} \cdots \zx_{\za_{r-1}} \, Q_{\za}$ with $Q_{\za} \in A[x]$ a homogenous monomial), we obtain $\im(d:\cK^{r-1}\to \cK^r )=\zx_1\cds\zx_r \,  (\sum_i A\lpq x\rpq \cdot x_i)$.
In conclusion,
$$H^r(\cK^{\bullet})=\;  \EnsQuot{\zx_1\cds\zx_r \cdot A\lpq x\rpq }{\, \zx_1\cds\zx_r (\sum_i A\lpq x\rpq \cdot x_i)} \simeq \; \zx_1\cds\zx_r \! \cdot \!A\;.$$
\end{proof}

\begin{rema}\label{degHK}
{\rm
a)
Proposition \ref{cohom} implies that
$H(\cK^{\bullet})=H^r(\cK^{\bullet}) \simeq \lpq \xi_1 \cds \zx_r\rpq  \cdot A$,
hence is a free  $A$-module of total rank $1$.
The degree of the basis element is $$\degr{\xi_1 \cds \zx_r}= \sum_{i=1}^{q} (\zg_i  +\zp)r_i +\sum_{i=q+1}^N \zg_i r_i = r'\zp +\sum_{i=1}^N \zg_i r_i \;.$$

b)
The result is independent of the chosen parity functor $\zP$ in the cochain complex $\cK^{\bullet}$.
}
\end{rema}
\medskip

A degree-preserving automorphism of $M$, $\zf \in \sss{Aut}^0(M)$ (represented by a matrix $T\in \sss{GL}^0(\vect{r};A)$),
naturally induces two automorphisms
\beas
\zf^{\zP}: \zP M \to \zP M & \mbox{ and } & (\zf^{-1})^{*}:M^{*}\to M^{*} \;,
\eeas
and hence an automorphism (of $A$-algebras)
\beas
\zvf\; : \; S_A^{\bullet}(\zP M \opl M^*) & \to & S_A^{\bullet}(\zP M \opl M^*) \;.
\eeas
on the ``total space'' $\cK:= \bigoplus_k\cK^k= S_A^{\bullet}(\zP M\oplus M^*)$ of the corresponding complex.
Explicitly, it is given by
\bea
\zvf(x_i)=
    \begin{cases}
    (\zf^{*})^{-1}(x_{i}) & 1\leq i \leq r' \\
    \zf^{\zP}(x_{i}) & r'< i \leq r
    \end{cases}
&\mbox{ and }&
\zvf(\zx_i)=
    \begin{cases}
        \zf^{\zP}(\zx_{i}) & 1\leq i \leq r' \\
         (\zf^{*})^{-1}(\zx_{i})& r'< i \leq r
    \end{cases}
\label{groupAction}\eea
i.e. corresponds to matrix multiplication by $T$ on $\zP M$ and by  $\gt(T^{-1})$ on $M^*$.

The differential $d:\cK \to \cK$ is invariant under this transformation
 (see Proposition \ref{differential}).
Hence, we obtain an automorphism on $\EnsQuot{\ker d}{\, \im d}$. 
This latter module, is equal to
$$\EnsQuot{\ker d}{\, \im d}= \EnsQuot{\Big( {\textstyle \bigoplus_k } \ker d\vert_{\cK^k} \Big)}{\; \Big( \bigoplus_k  \im d\vert_{\cK^{k-1}} \Big)} = \bigoplus_k \lp \EnsQuot{\ker d\vert_{\cK^k}}{\,\im d\vert_{\cK^{k-1}}}\rp = \bigoplus_k H^k(\cK^{\bullet})$$
and hence is just $H^r(\cK^{\bullet})$, thanks to Proposition \ref{cohom}.

By means of a graded matrix $T\in \sss{GL}^0(\vect{r};A)$ representing $\zf$, the obtained map rewrites as a group action of $\sss{GL}^0(\vect{r};A)$ on $H^r(\cK^{\bullet})$.
In other words, we have a group morphism
\begin{equation}
\label{Phi}
 \zF: \sss{GL}^0(\vect{r};A)\to \iAut^0\lp H^r(\cK^{\bullet})\rp \simeq (A^0)^\ts
\end{equation}
given  by
$$
\zF(T)\lp[\zx_1 \cdots \zx_r]\rp=\lpq \zvf(\zx_1 \cdots \zx_r)\rpq\,.
$$
We will now prove that this morphism coincides with the Graded Berezinian.

\begin{prop}\label{propGber}
 For all $T\in \sss{GL}^0(\vect{r};A)$, $\zF(T) $ is the operator of right multiplication by $\gber(T)$.
\end{prop}
\begin{proof}
We will first explicit in detail the proof for the super case (i.e. with grading group $\Z_2$),
following the description given in \cite{Man88}.
 Hence, in this case the graded Berezinian and the graded trace reduce to the classical Berezinian and the supertrace.

Let us consider two particular types of transformations $\zf$.
\begin{enumerate}
\item Let $\zf$ be a diagonal transformation, i.e. the corresponding graded matrix  is block-diagonal.
$$  T=\lp\begin{array}{cc} A&0\\0&B \end{array}\rp\in \sss{GL}^0\lp\vect{r}=(r',r'');A\rp $$
The matrix corresponding to the inverse dual is also block-diagonal
$$ ^{st}\! (T^{-1})=\lp\begin{array}{cc} ^{t}\! A^{-1}&0\\0&^{t}\! B^{-1} \end{array}\rp\in \sss{GL}^0(\vect{r};A) $$
Let us denote $a^{i}_{\;j}$ the entries of $A$ and $\f{b}^{\,\;i}_{j}$ the entries of $^{t}\!B^{-1}$.

We have that
\beas
\zvf(\zx_i)=\begin{dcases} \zf^{\zP}(\zx_i)=\sum\nolimits_{1\leq j \leq r'} \zx_{j} a^j_{\;i} & \mbox{ for } 1\leq i \leq r'\\
(\zf^{*})^{-1}(\zx_i) =\sum\nolimits_{r'< j \leq r} \zx_{j} \f{b}^{\; \, i-r'}_{ j-r'}  & \mbox{ for } r'< i \leq r \end{dcases}
\eeas
so that
\beas
\zvf(\zx_1\cdots \zx_r)&=&\zf^{\zP}(\zx_{1})\cdots \zf^{\zP}(\zx_{r'}) \cdot (\zf^{*})^{-1}(\zx_{r'+1}) \cdots (\zf^{*})^{-1}(\zx_{r}) \\
&=& \zx_{1}\cdots \zx_{r} \; \Bigg(\sum_{\zs \in \f{S}_{r'}}\op{sign}\zs \; a^{\zs(1)}_{\;\,1}\cds a^{\zs(r')}_{\;\,r'} \Bigg) \Bigg(\sum_{\zs \in \f{S}_{r-r'}}\op{sign}\zs \; \f{b}_{\zs(1)}^{\;1}\cds \f{b}_{\zs(r-r')}^{\;r-r'}\Bigg) \\
&=& \zx_{1}\cdots \zx_{r}\; \det(A)\det(B^{-1}) \\
&=&\zx_{1}\cdots \zx_{r}\;\ber(T) \; .
\eeas
 Here, signs appears because the elements of the subset $\{\zx_i\}_{1\leq i \leq r'}$ (respectively, $\{\zx_i\}_{r'< i \leq r}$) are of the same odd degree, hence
anticommute.
\item Let $\zf$ be a unitriangular transformation, i.e. the corresponding graded matrix is block-unitriangular. The value of its Berezinian then equals $1$.We will consider only the case of an upper unitriangular matrix, the case of a lower unitriangular matrix being similar.
Let $$T=\lp\begin{array}{cc} \I &C\\0&\I \end{array}\rp\in \sss{GL}^0(\vect{r};A).$$ Then the corresponding dual inverse is also block-unitriangular, more precisely $$^{st}\!( T^{-1})= \lp\begin{array}{cc} \I &0 \\ ^t\!C&\I \end{array}\rp\;.$$
We have in this case
\beas
\zvf(\zx_i)=\begin{cases} \zf^{\zP}(\zx_i)=  \zx_i & \mbox{ for } 1\leq i \leq r'\\
(\zf^{*})^{-1}(\zx_i) = \zx_i & \mbox{ for } r'< i \leq r \end{cases}
\eeas
so that
$$ \zvf(\zx_1\cdots \zx_r)= \zx_{1}\cdots \zx_{r} = \zx_{1}\cdots \zx_{r} \; \ber(T). $$
\end{enumerate}

Hence, we have proved that $\zF$ coincide with left multiplication by $\ber$ on block diagonal and block unitriangular matrices. By the uniqueness result concerning the  Berezinian, this suffices to conclude.

\medskip
This strategy of proof generalizes to the case of grading group $(\Z_2)^n$ for higher $n \in \N$, thanks to the analogous uniqueness result of the graded Berezinian (see section \ref{GBER}). We hence only have to verify 1. and 2. in this multigraded case.

Let $M$ be, as usual, a free module of rank $\vect{r}=(r_1,r_2,\ldots , r_N)$. The odd elements $\zx_i$ are by construction ordered by degree, so that in each subset  $\{\zx_i\}_{\sum_{\za<\zg}r_{\za}< i \leq \sum_{\za\leq \zg}r_{\za}}$ the elements are of the same odd degree, hence they anticommute. This implies that in the first step the expected signs (and hence the determinants) appears, as in the super case.
In the second step, by definition of the graded transpose (see section \ref{Sec2}), we still have that if $T$ is a block upper (resp. lower) unitriangular matrix then $\gt(T^{-1})$ is lower (resp. upper) block unitriangular. Let us consider, for simplicity, $n=2$ and $\vect{r}=(1,1,1,1)$. We then have only four odd $\zx_i$, of two different degrees $(0,1)$ and $(1,0)$. For a graded matrix
$$ T= \lp\begin{array}{cccc} 1&a&\star& \star\\ &1&\star&\star \\ & & 1& b \\ & & & 1 \end{array}\rp \; \in \sss{GL}^0\lp (1,1,1,1);A \rp$$
we have
$$\gt( T^{-1})= \lp\begin{array}{cccc} 1& & &\\ -a&1& & \\ \star'&\star' & 1&  \\ \star'&\star' & b & 1 \end{array}\rp $$
We then obtain
$$\begin{array}{rclrcl}
\zf^{\zP}(\zx_1)&=&  \zx_1 \;,\qquad &\qquad (\zf^{*})^{-1}(\zx_3)&=& \zx_3 + \zx_4 b \;,\\
 \zf^{\zP}(\zx_2) & = & \zx_1 a + \zx_2 \;,\qquad & \qquad (\zf^{*})^{-1}(\zx_4)&=& \zx_4\;,
\end{array}$$
so that
$$ \zvf(\zx_1\cdots \zx_4)=\zx_1(\zx_1 a + \zx_2)(\zx_3 + b\zx_4)\zx_4 = \zx_1 \cdots \zx_4 $$
since the $\zx_i$-s square to zero.

This clearly generalize to arbitrary $\vect{r}$ and arbitrary $n$.
\end{proof}

We summarize the above statements as follows.
\begin{thm}
The map
\beas
\zc : &  \gber(M)  \; \raa & \;H(\cK^{\bullet})\\
			& \;\, \op{B}(\{e_i\}) \; \mapsto & [\zx_1 \cds \zx_r]
\eeas
 is an $A$-module isomorphism of degree $r'\zp$.
\end{thm}

The degree of the isomorphism $\zc$ can be easily understood. Indeed, the $A$-module $\gber(M)$ is then either purely even or purely odd, and this depends only on the parity of  $\sum_{i=r'+1}^N \zg_i r_i$. This corresponds, when $n=1$, exactly to the well-known situation of the classical Berezinian.

\section{Cohomological Definition of the Graded Trace}\label{Sec4}

As we have seen in the previous section, the assignment
$\sss{GL}^0(\vect{r};A) \ni T \mapsto  \zvf_T \in \Aut^0\lp \cK \rp$,
where $\zvf_T$ is an automorphism of $A$-algebra corresponding to matrix multiplication by $T$ on $\zP M$ and by $\gt (T^{-1})$ on $M^*$,
defines  a group action for which the differential $d$ is invariant. It induces a group morphism (\ref{Phi}) which coincides with right multiplication by $\gber$.

In this section, we consider analogously the action of the colored Lie algebra of infinitesimal automorphisms
on the complex $\cK$.
We obtain the graded trace from the action on $\ber(M)$.
\medskip

\subsection{General Construction of the Action}

Consider, as before, the graded-commutative algebra $\cK = S_A^{\bullet}(\zP M \opl M^*)$, where $M$ is a free $A$-module  of rank $\vect{r}$ over a graded-commutative algebra $A$.
To any homogeneous square matrix $S \in \sss{M}(\vect{r};A)$ of degree $\degr{S}$ we can associate a graded derivation of the same degree $L_S \in \op{Der}^{\degr{S}}(\cK)$.
$L_S$ is given by matrix multiplication by $S$ on $M$ and by matrix multiplication by $-\gt\! S$ on~$M^*$.
Note that since $\cK$ has also an $A$-module structure, it is natural to restrict ourselves only to derivations that are also $A$-module morphisms, i.e.
$$ \op{Der}_A(\cK) := \{ D \in \op{Der}\lp \cK\rp \; : \; D(a)=0, \; \forall a\in A \} \subset \iEnd_A(\cK) \;.$$

Hence, more explicitly, we have an assignment
\bea \label{algaction}
L:  \sss{M}(\vect{r};A)\ni S \mapsto L_S \in \op{Der}_A\lp \cK\rp
\eea
where, for any homogeneous matrix $S$, $L_S \in \iEnd_A(\cK)$ is given on generating elements by
\begin{align}
L_S(\zP e_i)&= \sum_k \zP e_k \, s^k_{\,\;i} \tag{\theequation a} \label{Ls1} \\
\mbox{and} \qquad L_S(\ze^i)&= - \sum_k \ze^k (-1)^{\la \degr{e}_i+\degr{e}_k, \degr{S}+\degr{e}_k \ra}s^i_{\,\;k}= -(-1)^{\la \degr{e}_i,\degr{S}\ra}\sum_k s^i_{\,\;k}\,\ze^k \tag{\theequation b}\label{Ls2}
\end{align}
and extends to arbitrary elements by the graded Leibniz rule
$$ L_S(ab)=L_S(a)\, b +(-1)^{\la \degr{S}, \degr{a} \ra} a\, L_S(b) \;, \quad \forall a,b \in \cK  \,.$$

In fact, $L$ is a colored Lie algebra morphism of degree $0$. Indeed,  we see that  the equality
$$ [L_S,L_T]=L_{[S,T]} $$
holds on $M$ by construction, and on $M^*$ it follows from (\ref{lem}) (Corollary \ref{lemma}, p. 8).

\medskip

\subsection{Deducing the Graded Trace}

The second main result of this paper is as follows.

\begin{thm}
\label{MainTwo}
Given an even matrix $S\in \sss{M}\evp(\vect{r};A)$, its action (by derivation) on the cohomology $H(\cK^{\bullet})$
is well-defined provided one of the following conditions is satisfied:
\begin{enumerate}
\item \label{case1}
$S \in \sss{M}^0(\vect{r};A)$ and the parity
of $\zP$ is an arbitrary odd element $\zp \in (\Z_2)^n \odp $;
\item \label{case2}
$S $ is an arbitrary even matrix, $n$ is an odd integer
and
$$
\zp =(1,1,\ldots ,1)\in (\Z_2)^n \odp.
$$
\end{enumerate}
In both cases, the action of $S$ coincides with the operator of left multiplication by~$\gtr(S)$.
\end{thm}

To prove the theorem, let us first determine the conditions sufficient for invariance of
 the differential $d:\, \cK \to \cK$ under the action 
 (\ref{algaction}).

\begin{lem}\label{propdinv}
$ [L_S,d]=0 $ if one of the above conditions \ref{case1} or \ref{case2} is satisfied.
\end{lem}
    \begin{proof}
    By definition, the operator $d$ is left multiplication by $\sum_i \zP e_i\, \ze^i$. Hence,  $[L_S,d]=0$ is equivalent to $L(\sum_i \zP e_i\, \ze^i)=0$.

    From (\ref{Ls1}-b), we have that
    \beas
L_S\Big( \sum_i \zP e_i\, \ze^i\Big)&=&
    \sum_i \lp L_S(\zP e_i)\, \ze^i + (-1)^{\la \degr{S}, \degr{e}_i+\zp\ra} \zP e_i\, L_S(\ze^i)\rp\\
    &=& \sum_{i,k}  \zP e_k \, s^k_{\,\;i}\, \ze^i - (-1)^{\la \degr{S}, \zp\ra} \sum_{i,k} \zP e_i\, s^i_{\,\;k}\ze^k \\
    &=& \sum_{u,v} \zP e_u \, s^u_{\;\,v} \, \ze^v   \lp 1-(-1)^{\la \degr{S},\zp \ra}\rp
\eeas
which is equal to zero if and only if
\be \label{condinv} \la \degr{S},\zp \ra=0\;. \ee
In particular this holds whenever $\degr{S}=(0,0,\ldots ,0)$
(then $\zp$ is arbitrary).
On the other hand, assuming that $n$ is an odd integer and $\zp=(1,1,\ldots ,1)$,
the equality (\ref{condinv}) then also holds for every 
homogeneous matrix of even (not necessarily zero) degree.
    \end{proof}

Consequently, in the two cases of the above proposition, $L$ induces an action (respectively denoted by $\cL^{(0)}$ and $\cL$) on the cohomology $ H(\cK^{\bullet})=H^r(\cK^{\bullet})$.
In other words, we have in the first case an algebra morphism
$$
 \cL^{(0)}: \sss{M}^0(\vect{r};A)\to \iEnd^0\lp H^r(\cK^{\bullet})\rp \simeq A^0
$$
and in the second case (i.e. when $n$ is odd) an algebra morphism
$$
 \cL: \sss{M}\evp(\vect{r};A)\to \iEnd\evp\lp H^r(\cK^{\bullet})\rp \simeq A\evp   \;,
$$
both given  by
$$
L_S\lp [\zx_1 \cdots \zx_r] \rp := [L_S\lp \zx_1 \cdots \zx_r \rp]\;.
$$
Here, we use again the $(x,\zx)$'s notation introduced in (\ref{changenotation}). With this notation, (\ref{Ls1}-b) give in particular
\[
 L_S(\zx_i)=\begin{dcases}
                 \sum\nolimits_{k=1}^{r'} \zx_k s^{k}_{\;\;i} + \sum\nolimits_{k=r'+1}^{r} x_k s^{k}_{\;\;i}  & \quad \mbox{ if } 1\leq i \leq r' \\
                  (-1)^{\la \degr{S}+\zp,\degr{e}_i \ra} \lp \sum\nolimits_{k=1}^{r'}  s^{k}_{\;\;i}x_k - \sum\nolimits_{k=r'+1}^{r} (-1)^{\la \degr{e}_k,\zp\ra} s^{k}_{\;\;i}\zx_k \rp  & \quad \mbox{ if } r'< i \leq r \\
\end{dcases}
\]
so that
\beas
L_S\lp [\zx_1 \cdots \zx_r] \rp &=& [L_S\lp \zx_1 \cdots \zx_r \rp]\\
    &=& \lpq \sum_{i=1}^r (-1)^{\la \degr{S}, \sum_{j<i} \degr{\zx}_j \ra} \zx_1 \cdots \zx_{i-1}L_S(\zx_i)\zx_{i+1}\cdots \zx_r \rpq  \nonumber \\
    &=& \sum_{i=1}^{r'}\sum_{k=1}^{r'} (-1)^{\la \degr{S}, \sum_{j<i}  \degr{\zx}_j \ra}  [ \zx_1 \cdots \zx_{i-1}\zx_k s^k_{\;\,i}\zx_{i+1}\cdots \zx_r ] \nonumber \\
    & &  + \sum_{i=1}^{r'} \sum_{k=r'+1}^{r} (-1)^{\la \degr{S}, \sum_{j<i}  \degr{\zx}_j \ra}  [ \zx_1 \cdots \zx_{i-1}x_k s^k_{\;\,i}\zx_{i+1}\cdots \zx_r ] \quad \nonumber \\
    & &  + \sum_{i=r'+1}^{r}\sum_{k=1}^{r'} (-1)^{\la \degr{S}, \sum_{j<i}  \degr{\zx}_j \ra+\la \degr{S} +\zp,\degr{e}_i \ra }  \lpq \zx_1 \cdots \zx_{i-1} s^i_{\;\,k} x_k \zx_{i+1}\cdots \zx_r \rpq   \nonumber \\
    & &  - \sum_{i=r'+1}^{r}\sum_{k=r'+1}^{r} (-1)^{\la \degr{S}, \sum_{j<i}  \degr{\zx}_j \ra+\la \degr{S}+\zp,\degr{e}_i \ra + \la \degr{e}_k, \zp \ra}  \lpq \zx_1 \cdots \zx_{i-1} s^i_{\;\,k} \zx_k \zx_{i+1}\cdots \zx_r \rpq  \nonumber \\
    &=& \sum_{i=1}^{r'} (-1)^{\la \degr{S}, \sum_{j<i} \degr{\zx}_j \ra} \lpq \zx_1 \cdots\zx_{i-1}\zx_i s^i_{\,\;i}\zx_{i+1}\cdots \zx_r\rpq  \nonumber \\
    & &   - \sum_{i=r'+1}^r (-1)^{\la \degr{S}, \sum_{j<i} \degr{\zx}_j \ra + \la \degr{S},\degr{e}_i\ra} \lpq\zx_1 \cdots \zx_{i-1}s^i_{\,\;i} \zx_i \zx_{i+1}\cdots \zx_r\rpq \nonumber  \\
    &=& \sum_{i=1}^{r'} (-1)^{\la \degr{S},  \degr{\zx}_i \ra} s^i_{\,\;i} \lpq \zx_1 \cdots \zx_r \rpq - \sum_{i=r'+1}^r (-1)^{\la \degr{S}, \degr{e}_i \ra} s^i_{\,\;i} \lpq \zx_1 \cdots \zx_r \rpq \nonumber \\
    &=& \sum_{i=1}^{r'} (-1)^{\la \degr{S},  \degr{e}_i+\zp \ra +\la \degr{e}_i,\degr{e}_i \ra} s^i_{\,\;i} \lpq \zx_1 \cdots  \zx_r\rpq  +\sum_{i=r'+1}^r (-1)^{\la \degr{S}, \degr{e}_i \ra+\la \degr{e}_i,\degr{e}_i \ra} s^i_{\,\;i} \lpq \zx_1 \cdots \zx_r\rpq \label{LYgtr}
\eeas
Clearly, if we are in one of the two cases described in Theorem \ref{MainTwo}, this rewrites as
\beas L_S\lp[ \zx_1 \cdots \zx_r] \rp =\lp \sum_{i=1}^r (-1)^{\la\degr{e}_i+\degr{S},\degr{e}_i\ra}s^i_{\,\;i} \rp [\zx_1 \cdots \zx_r]=  \gtr(S) \, [\zx_1 \cdots \zx_r] \;.\eeas

Theorem \ref{MainTwo} is proved.



\end{document}